\documentclass[12pt]{amsart} \usepackage{amssymb} \usepackage{mathabx}

\newcommand{\wrlab}[1]{\label{#1}}

\newtheorem{thm}{THEOREM}[section]

\newtheorem{lem}[thm]{LEMMA} 
\newtheorem{cor}[thm]{COROLLARY} 
 
 \newtheorem{thm*}{THEOREM}[]

\newcommand{\tref}[1]{Theorem~\ref{#1}}
\newcommand{\cref}[1]{Corollary~\ref{#1}}

\newcommand{\lref}[1]{Lemma~\ref{#1}}

\def\N{{\mathbb N}} \def\Z{{\mathbb Z}} 
\def\R{{\mathbb R}}

  \def\scrb{{\mathcal B}}
 \def\scrf{{\mathcal F}} 
 \def\scrk{{\mathcal K}} 
 \def\scrp{{\mathcal P}}

\def\scrz{{\mathcal Z}}  \def\scrc{{\mathcal C}}
\def\scrn{{\mathcal N}} 

\def\bfu{{\bf1}}

\font\tenolde=eufm10 at 10pt
\font\sevenolde=eufm7
\font\fiveolde=eufm5
\newfam\oldefam
\textfont\oldefam=\tenolde
\scriptfont\oldefam=\sevenolde
\scriptscriptfont\oldefam=\fiveolde

  \def\Lg{{\mathfrak g}}

\def\id{{\rm id}}

\def\dim{\hbox{dim\,}}

\def\exp{\hbox{exp}}

\def\Stab{\hbox{Stab}}

\begin{document}
\bibliographystyle{alpha}


\title[Generic freeness for $C^\infty$ actions]
{Generic freeness in frame bundle prolongations of $C^\infty$ actions}


\author{Scot Adams}
\address{School of Mathematics\\ University of Minnesota\\Minneapolis, MN 55455
\\ adams@math.umn.edu}

\date{May 24, 2017\qquad Printout date: \today}

\begin{abstract}
Let a real Lie group $G$ act on a $C^\infty$ real manifold~$M$.
Assume that the action is $C^\infty$ and that
every nontrivial element of~$G$ has a nowhere dense fixpoint set in $M$.
We show that, in~some higher order frame bundle $F$ of~$M$,
there exists a $G$-invariant meager subset $Z$ of $F$ such that
the $G$-action on $F\backslash Z$ is free.
A similar result holds for submanifold jet bundles.
\end{abstract}

\maketitle

 

\section{Introduction\wrlab{sect-intro}}

P.~Olver's freeness conjecture (in his words) asserts: ``If a Lie
group acts effectively on a manifold, then, for some $n<\infty$, the
action is free on [a nonempty] open subset of the jet bundle of order $n$.''
For {\it local} freeness, see
Theorem 6.14 of \cite{olver:movfrmsing},
Theorem 12.1 of~\cite{adams:locfreeacts}
and \tref{thm-loc-free-generic} below.
For $C^\infty$ actions, the freeness question 
was resolved in \cite{adams:cinftyctx} by a countexample for~the additive group $\Z$ of integers.
Then, in~\cite{adams:comegafreeness},
we exposed a $C^\omega$~counterexample, also an action of $\Z$.
Using Lemma 17.3 of~\cite{adams:cinftyctx},
and following the proof of~Theorem 18.1 in~\cite{adams:cinftyctx},
we see that this $C^\omega$ action of $\Z$ induces to a $C^\omega$~counterexample for any
connected Lie group with~noncompact center.
There is also a converse: By Corollary 3.2 of~\cite{adams:freeAbStab},
the frame bundle freeness conjecture holds
(in~its $C^\infty$ form and, consequently, in~its $C^\omega$~form)
for~any connected Lie group with compact center.
So, for many important Lie groups, the frame bundle freeness conjecture holds;
for many others, frame bundle freeness
(and, consequently, submanifold jet bundle freeness) fails.
It remains open whether a Lie group exists
for which the submanifold jet bundle freeness conjecture holds.

For general Lie groups, we now turn our attention to {\it generic} freeness.
In Theorem 20 of \cite{adolv:genericframe},
we proved that any effective $C^\omega$ action of~a~connected real Lie group
on a connected $C^\omega$ real manifold has a~frame bundle prolongation
that is generically free in some nonempty open subset of the total space.
In \tref{thm-genfree-frame} below,
we generalize this result to $C^\infty$ actions
of arbitrary Lie groups on arbitrary $C^\infty$ real manifolds
and prove generic freeness on the entire total space.

In Theorem 26 of \cite{adolv:genericframe},
we proved a $C^\omega$ generic freeness result for~submanifold jet bundles,
though only for connected groups,
and only on a nonempty open set in the total space.
In \cref{cor-genfree-jet} below,
we improve that to $C^\infty$, to arbitrary groups, and to the entire total space.

The results in Sections 3-7 are elementary.
Our main results are \tref{thm-ordone-stabs-nontriv}
and \tref{thm-genfree-frame}.

\section{Global notation and conventions\wrlab{sect-global}}

Let $\N:=\{1,2,3,\ldots\}$ and $\N_0:=\N\cup\{0\}$.
For any set~$S$, we denote the number of elements in $S$ by $\#S\in\N_0\cup\{\infty\}$.
For any sets $S$~and~$T$, for any function $f:S\to T$,
for any $T_0\subseteq T$, we denote the $f$-preimage of $T_0$ by
$f^*(T_0):=\{s\in S\,|\,f(s)\in T_0\}$.
For any set $S$, the identity map $\id_S:S\to S$ is defined by $\id_S(x)=x$.

Throughout this paper, by ``group'', we mean ``multiplicative group'',
unless otherwise specified.
By ``manifold'', we mean ``$C^\infty$ real manifold'',
unless otherwise specified.
By ``vector space'', we mean ``real vector space'',
unless otherwise specified.
By ``Lie group'', we mean ``real Lie group'',
unless otherwise specified.
On any principal bundle, the structure group action is
denoted as a right action; all other group actions are from the left.
Throughout this paper, a subset of~a~topological space acquires, without comment,
its relative topology, inherited from that ambient topological space.
An open subset of a manifold acquires, without comment,
its relative manifold structure, inherited from the ambient manifold.
Every finite dimensional vector space is given, without comment,
its standard topology and manifold structure.
For any group $G$, the identity element of $G$ is denoted~$1_G$.
For any~vector space~$V$, the zero element of~$V$ is denoted~$0_V$.
For all~$n\in\N$, let $0_n:=0_{\R^n}$;
then $0_n=(0,\ldots,0)\in\R^n$.
For any~action of~a group $G$ on a set $S$,
for any $A\subseteq G$, for any $x\in S$,
we define $\Stab_A(x):=\{g\in A\,|\,gx=x\}$.
An action of~a topological group~$G$ on a set $S$ is {\bf locally free} if,
for all~$x\in S$, we have: $\Stab_G(x)$ is discrete.

Let a group $G$ act on a topological space $X$.
The $G$-action on $X$ is~{\bf fixpoint rare} if,
for all $g\in G\backslash\{1_G\}$,
the set $\{x\in X\,|\,gx=x\}$ has empty interior in $X$,
{\it i.e.}, if no nontrivial element of $G$ fixes
every point in a nonempty open subset of~$X$.
If $G$ is a connected locally compact Hausdorff topological group,
if $X$ is locally compact Hausdorff
and if~the $G$-action on $X$ is continuous,
then, in Lemma 6.1 of~\cite{adams:locfreeacts},
we~show that the $G$-action on $X$ is fixpoint rare iff,
for every $G$-invariant nonempty open subset~$V$ of $X$,
the $G$-action on $V$ is~effective.

Let $M$ be a manifold.
We will denote the tangent bundle of $M$ by~$TM$;
a $C^\infty$ action of a Lie group $G$ on $M$ induces, by differentiation, a~$G$-action on $TM$.
For any $x\in M$, the tangent space to $M$ at~$x$ is~denoted $T_xM$.
For any $k\in\N_0$,
let $\scrf^kM$ denote the $k$th order frame bundle of $M$;
a $C^\infty$ action of a Lie group $G$ on $M$ induces, by~differentiation, a.k.a.~``prolongation'',
a (left) $G$-action on~$\scrf^kM$,
commuting with the (right) action of the structure group.
For any integer $k\ge0$, for all $x\in M$, the fiber in~$\scrf^kM$ over~$x$ is denoted $\scrf^k_xM$.

Let $M$ be a manifold.
Let $d:=\dim M$ and let $i,j\in\N_0$.
Both $\scrf^{i+j}M$~and $\scrf^i\scrf^jM$ are fiber bundles over $\scrf^jM$.
Let $\scrp$ be the vector space of polynomial functions $\R^d\to\R^d$ of degree $\le j$.
Let $\varepsilon_1,\ldots,\varepsilon_d$ be the standard basis of $\R^d$.
Define
$$\Gamma\,\,:=\,\,\{\,(\alpha_1,\ldots,\alpha_d,\delta)\in\N_0^{d+1}\,\,|\,\,\alpha_1+\cdots+\alpha_d\le j\,\,,\,\,1\le\delta\le d\,\}.$$
For every $\gamma=(\alpha_1,\ldots,\alpha_d,\delta)\in\Gamma$,
let the function $P_\gamma\in\scrp$ be defined by~$P_\gamma(x_1,\ldots,x_d)=x_1^{\alpha_1}\cdots x_d^{\alpha_d}\varepsilon_\delta$.
Let $\scrb:=\{P_\gamma\,|\,\gamma\in\Gamma\}$.
Then $\scrb$ is a basis of $\scrp$.
Then $\dim\scrp=\#\scrb=\#\Gamma$.
The lexicographic ordering on $\Gamma$ corresponds to an ordering of the basis $\scrb$.
The resulting ordered basis of~$\scrp$ gives rise to a vector space isomorphism $\kappa:\scrp\to\R^{\#\Gamma}$.
Following~\S10 of \cite{adams:locfreeacts},
we construct a map $\Phi_{M,\kappa}^{i,j}:\scrf^{i+j}M\to\scrf^i\scrf^jM$
of fiber bundles over $\scrf^jM$.
In~this paper, we define $\Psi_M^{i,j}:=\Phi_{M,\kappa}^{i,j}$.
By~Lemma 10.1 of~\cite{adams:locfreeacts},
$\Psi_M^{i,j}:\scrf^{i+j}M\to\scrf^i\scrf^jM$ is natural in $M$.

\section{Miscellaneous results\wrlab{sect-misc-results}}

\begin{lem}\wrlab{lem-stababovebelow}
Let a Lie group $G$ have a $C^\infty$ action on a manifold $M$.
Let $k\in\N$, $X:=\scrf^kM$, $W:=\scrf^{k-1}M$, $x\in X$.
Let~$\sigma:X\to W$ be the canonical map, $w:=\sigma(x)$, $v\in T_wW$.
Then $\Stab_G(x)\subseteq\Stab_G(v)$.
\end{lem}

\begin{proof}
Let $g\in\Stab_G(x)$ be given.
We wish to show that $g\in\Stab_G(v)$.

Let $y:=\Psi_M^{1,k-1}(x)\in\scrf^1_wW$.
Let $r:=\dim T_wW$.
Identify $y\in\scrf^1_wW$ with an ordered basis of $T_wW$,
and then fix $u_1,\ldots,u_r\in T_wW$ such that $y=(u_1,\ldots,u_r)$.
Then $\{u_1,\ldots,u_r\}$ is a basis for $T_wW$,
and so, as $v\in T_wW$, choose $c_1,\ldots,c_r\in\R$ such that $v=c_1u_1+\cdots+c_ru_r$.
Since $g\in\Stab_G(x)$, we see that $gx=x$.
By naturality, the map $\Psi_M^{1,k-1}:X\to\scrf^1W$ is $G$-equivariant.
So, since $gx=x$, we get $gy=y$.
So, since $y=(u_1,\ldots,u_r)$,
we get $(gu_1,\ldots,gu_r)=(u_1,\ldots,u_r)$.
Then
$gv=g(c_1u_1+\cdots+c_ru_r)
=c_1(gu_1)+\cdots+c_r(gu_r)
=c_1u_1+\cdots+c_ru_r=v$,
and so $g\in\Stab_G(v)$, as desired.
\end{proof}

\begin{lem}\wrlab{lem-immersiveness}
Let a Lie group $G$ have a $C^\infty$ action on a manifold~$M$.
Let $p\in M$, $a\in G$.
Assume that $\Stab_G(p)$ is discrete.
Define $f:G\to M$ by $f(g)=gp$.
Then $(df)_a:T_aG\to T_{ap}M$ is injective.
\end{lem}

\begin{proof}
For all $x\in M$, let $G_x:=\Stab_G(x)$.
Then $G_p$ is discrete,
so $aG_p$~is discrete,
so $T_a(aG_p)$ is trivial.
Also, by Lemma 4.1 of \cite{adams:locfreeacts}
(with $\overline{p}$ replaced by~$f$), $T_a(aG_p)=\ker[(df)_a]$.
Then $\ker[(df)_a]$ is trivial,
and so $(df)_a:T_aG\to T_{ap}M$ is injective, as desired.
\end{proof}

We make a small improvement to Theorem 12.1 of \cite{adams:locfreeacts}:

\begin{thm}\wrlab{thm-loc-free-generic}
Let a Lie group $G$ have a $C^\infty$ action on a manifold~$M$.
Let $G^\circ$ be the identity component of $G$.
Assume the $G^\circ$-action on $M$ is fixpoint rare.
Let $n:=\dim G$.
Let $\ell\in\N_0$, and assume that~$\ell\ge n-1$.
Let $\pi:\scrf^\ell M\to M$ denote the structure map of~$\scrf^\ell M$.
Then there exists a $G$-invariant dense open subset~$M_1$ of $M$
such that the $G$-action on~$\pi^*(M_1)$ is locally free.
\end{thm}

\begin{proof}
If $n=0$, then $G$ is discrete, so every action of $G$ is locally free,
and we may set $M_1:=M$.
We assume $n\ge1$.

By Theorem 12.1 of \cite{adams:locfreeacts},
choose a $G$-invariant dense open subset~$Q$ of $\scrf^{n-1}M$
s.t.~the $G$-action on~$Q$ is locally free.
Let $\sigma:\scrf^\ell M\to\scrf^{n-1}M$ be the canonical map and let $R:=\sigma^*(Q)$.
Since $\sigma:\scrf^\ell M\to\scrf^{n-1}M$ is $G$-equivariant and continuous and open,
and since $Q$ is a $G$-invariant dense open subset of $\scrf^{n-1}M$,
we see that $R$~is a  $G$-invariant dense open subset of $\scrf^\ell M$.
Since $\sigma:\scrf^\ell M\to\scrf^{n-1}M$ is $G$-equivariant,
we see: $\forall x\in \scrf^\ell M$, $\Stab_G(x)\subseteq\Stab_G(\sigma(x))$.
So, since the~$G$-action on~$Q$ is locally free,
the~$G$-action on $R$ is also locally free.
Let $M_1:=\pi(R)$.
As $\pi:\scrf^\ell M\to M$ is $G$-equivariant
and as $R$~is $G$-invariant,
we see that $M_1$~is $G$-invariant.
As $\pi:\scrf^\ell M\to M$ is continuous, open and surjective,
and as $R$ is dense open in~$\scrf^\ell M$,
we see that $M_1$~is dense open in~$M$.
It remains to~show: the~$G$-action on $\pi^*(M_1)$ is locally free.

Let $H$ be the structure group of $\scrf^\ell M$.
For all $x\in \scrf^\ell M$, for all~$h\in H$, we have
$\Stab_G(xh)=\Stab_G(x)$.
So, since the $G$-action on $R$ is locally free,
it follows that the $G$-action on $RH$ is also locally free.
So, as $\pi^*(M_1)=\pi^*(\pi(R))=RH$,
the $G$-action on~$\pi^*(M_1)$ is locally free.
\end{proof}

\section{Flow boxes\wrlab{sect-flowboxes}}

\begin{lem}\wrlab{lem-flow-box-existence}
Let $G$ be a Lie group acting on a manifold~$M$.
Assume that the $G$-action on $M$ is $C^\infty$.
Let $e:=1_G$ and let~$p\in M$.
Assume that $\Stab_G(p)$ is discrete.
Then there exist
\begin{itemize}
\item open neighborhoods \quad $G_0$ in $G$ of $e$ \quad and \quad $M_0$ in $M$ of $p$,
\item a $C^\infty$ function $\gamma:M_0\to G_0$ \qquad and \qquad $\bullet$ a subset $T$ of $M$
\end{itemize}
such that: \qquad $p\in T\subseteq G_0T=M_0$, \quad $\gamma(p)=e$,
\begin{itemize}
\item[] \vskip-.05in
the map $(g,t)\mapsto gt:G_0\times T\to M_0$ is a bijection \qquad and
\item[] \qquad\qquad\quad $\forall x\in M_0,\,\,[(\gamma(x))^{-1}]x\in T$.
\end{itemize}
\end{lem}

The Frobenius Theorem is neither needed nor useful
in proving this.

\begin{proof}
Define $f:G\to M$ by $f(g)=gp$.
By \lref{lem-immersiveness} (with $a$~replaced by $e$),
$(df)_e:T_eG\to T_pM$~is injective.
Let $\Lg:=T_eG$, $W:=T_pM$, $\phi:=(df)_e$.
Then $\phi:\Lg\to W$ is injective.
Let $U:=\phi(\Lg)$.
Then $U$ is a vector subspace of $W$
and $\phi:\Lg\to U$ is a vector space isomorphism.
Then $\dim\Lg=\dim U$.
Let $d:=\dim M$, $n:=\dim G$, $c:=d-n$.
We have $\dim W=\dim M=d$.
Also, $\dim U=\dim\Lg=\dim G=n$.
Choose a vector subspace $V$ of $W$ such that $U\cap V=\{0_W\}$
and such that $\dim V=c$.
Let $V_0:=T_{0_c}\R^c$.
Then $\dim V_0=c=\dim V$.
Choose a~$C^\infty$~function $r:\R^c\to M$
such that $r(0_c)=p$
and $((dr)_{0_c})(V_0)=V$.
Let $\rho:=(dr)_{0_c}:V_0\to W$.
Then $\rho(V_0)=V$, so $\rho:V_0\to V$ is surjective.
So, since $\dim V_0=\dim V$,
$\rho:V_0\to V$ is a vector space isomorphism.

Define $s:G\times\R^c\to M$ by $s(g,y)=g[r(y)]$.
Let $q:=(e,0_c)\in G\times\R^c$.
Then $s(q)=e[r(0_c)]=ep=p$.
Let $\sigma:=(ds)_q:\Lg\oplus V_0\to W$.
For all~$g\in G$,
$s(g,0_c)=g[r(0_c)]=gp=f(g)$.
It follows, for all~$X\in\Lg$, that $\sigma(X,0_{V_0})=\phi(X)$.
For all $y\in\R^c$, $s(e,y)=e[r(y)]=r(y)$.
It follows, for all $Y\in V_0$, that $\sigma(0_\Lg,Y)=\rho(Y)$.

{\it Claim 1:}
The map $\sigma:\Lg\oplus V_0\to W$ is injective.
{\it Proof of Claim 1:}
We wish to show: $\ker\sigma=\{(0_\Lg,0_{V_0})\}$.
Let $X\in\Lg$ and $Y\in V_0$ be given,
and assume $\sigma(X,Y)=0_W$.
We wish to show: $X=0_\Lg$ and $Y=0_{V_0}$.

As we observed,
$\sigma(X,0_{V_0})=\phi(X)$ and $\sigma(0_\Lg,Y)=\rho(Y)$.
Then
$[\phi(X)]+[\rho(Y)]=[\sigma(X,0_{V_0})]+[\sigma(0_\Lg,Y)]=\sigma(X,Y)=0_W$.
It follows that $\phi(X)=\rho(-Y)\in\rho(V_0)=V$.
Also, $\phi(X)\in\phi(\Lg)=U$.
Then $\phi(X)\in U\cap V=\{0_W\}$.
Then $\phi(X)=0_W=\phi(0_\Lg)$, so, by injectivity of~$\phi$,
we conclude that $X=0_\Lg$.
It remains to show that $Y=0_{V_0}$.

We have $\rho(Y)=-[\phi(X)]=0_W=\rho(0_{V_0})$.
So, since $\rho$ is injective, we get $Y=0_{V_0}$, as desired.
{\it End of proof of Claim~1.}

We have $\dim(\Lg\oplus V_0)=(\dim\Lg)+(\dim V_0)=n+c=d=\dim W$.
By~Claim 1, $\sigma:\Lg\oplus V_0\to W$ is injective.
Then $\sigma:\Lg\oplus V_0\to W$ is a vector space isomorphism.
So, since $\sigma=(ds)_q$, by the Inverse Function Theorem,
choose an open neighborhood $N$ in~$G\times\R^c$ of~$q$
such that $s(N)$~is an open subset of $M$
and $s|N:N\to s(N)$ is a $C^\infty$ diffeomorphism.
Recall: $q=(e,0_c)\in G\times\R^c$, $s(q)=p\in M$.
Choose open neighborhoods~$G_0$ in $G$ of~$e$ and $T_0$ in $\R^c$ of~$0_c$
such that $G_0\times T_0\subseteq N$.
Then $G_0\times T_0$ is an open neighborhood in $G\times\R^c$ of $q$.
Let $M_0:=s(G_0\times T_0)$.
Then $M_0$ is an open neighborhood in $M$ of $p$.
Let $s_0:=s|(G_0\times T_0)$.
Then $s_0:G_0\times T_0\to M_0$ is a $C^\infty$ diffeomorphism.
For all $g\in G$, for all $y\in T_0$, $s_0(g,y)=s(g,y)=g[r(y)]$.
Let $T:=r(T_0)$.
Then $M_0=s(G_0\times T_0)=G_0[r(T_0)]=G_0T$.

{\it Claim 2:}
The map $(g,t)\mapsto gt:G_0\times T\to M_0$ is a bijection.
{\it Proof of Claim 2:}
Since $M_0=G_0T$, we need only prove injectivity.
Let $g,h\in G_0$ and $t,u\in T$ be given,
and assume that $gt=hu$.
We wish to~show: $g=h$ and $t=u$.
Since $t,u\in T=r(T_0)$,
choose $y,z\in T_0$ such that $t=r(y)$ and $u=r(z)$.
It suffices to show: $g=h$ and $y=z$.

We have
$s_0(g,y)=g[r(y)]=gt=hu=h[r(z)]=s_0(h,z)$.
So, by~injectivity of $s_0$, we get $g=h$ and $y=z$.
{\it End of proof of Claim~2.}

As $e\in G_0$, we have $T\subseteq G_0T$.
So, since $p=r(0_c)\in r(T_0)=T$ and $G_0T=M_0$, by Claim 2,
it remains only to show that there exists a~$C^\infty$~function $\gamma:M_0\to G_0$
such that:
\begin{itemize}
\item[]both \quad ( $\gamma(p)=e$ ) \quad and \quad (~$\forall x\in M_0,\,\,[(\gamma(x))^{-1}]x\in T$ ).
\end{itemize}

We have $s_0(q)=s(q)=p$,
so $s_0^{-1}(p)=q$.
Define $\pi:G_0\times T_0\to G_0$ by $\pi(g,y)=g$.
Then $\pi(q)=\pi(e,0_c)=e$.
Let $\gamma:=\pi\circ s_0^{-1}:M_0\to G_0$.
Then $\gamma(p)=(\pi\circ s_0^{-1})(p)=\pi(s_0^{-1}(p))=\pi(q)=e$.
Let $x\in M_0$ be given.
We wish to~prove that $[(\gamma(x))^{-1}]x\in T$.

Because $s_0:G_0\times T_0\to M_0$ is a surjective mapping,
choose $g\in G_0$ and $y\in T_0$ such that $s_0(g,y)=x$.
Then $s_0^{-1}(x)=(g,y)$.
It follows that $\gamma(x)=(\pi\circ s_0^{-1})(x)=\pi(s_0^{-1}(x))=\pi(g,y)=g$.
Then
\begin{eqnarray*}
&&[(\gamma(x))^{-1}]x\,\,=\,\,g^{-1}x=g^{-1}[s_0(g,y)]\\
&&\quad=\,\,g^{-1}g[r(y)]\,\,=\,\,r(y)\,\,\in\,\,r(T_0)\,\,=\,\,T,
\end{eqnarray*}
as desired.
\end{proof}

\section{A regularity result\wrlab{sect-regularity}}

\begin{lem}\wrlab{lem-regularity}
Let $G$ be a Lie group acting on a manifold~$M$.
Assume that the $G$-action on $M$ is $C^\infty$ and locally free.
Let $R$~be an open subset of $M$.
Let $\eta:R\to G$ be a continuous function.
Define $\overline{\eta}:R\to M$ by~$\overline{\eta}(x)=[\eta(x)]x$.
Assume $\overline{\eta}$ is $C^\infty$.
Then $\eta$ is~$C^\infty$.
\end{lem}

The idea of proof is: 
According to \lref{lem-flow-box-existence},
near any $p\in M$, the
$G$-manifold $M$ is locally $C^\infty$ $G$-diffeomorphic to $G_0\times T$.
These $G_0\times T$ coordinates allow us, near any $p\in R$, to recover
$\eta$ from $\overline{\eta}$ via a local $C^\infty$ formula.
So, since $\overline{\eta}$ is $C^\infty$,
we see that, near any $p\in R$, $\eta$ is $C^\infty$.

\begin{proof}
Let $e:=1_G$.
Let $p\in R$ be given.
We wish to show that there exists an open neighborhood $U$ in $R$ of $p$
such that $\eta$ is $C^\infty$ on $U$.

Let $c:=\eta(p)$.
Define $\delta:R\to G$ by $\delta(x)=c^{-1}[\eta(x)]$.
For all~$x\in R$, we have $\eta(x)=c[\delta(x)]$.
It therefore suffices to show that there exists an open neighborhood $U$ in $R$ of $p$
such that $\delta$ is $C^\infty$ on $U$.

We have $\delta(p)=c^{-1}[\eta(p)]=c^{-1}c=e$.
Since $\eta$~is continuous, $\delta$~is continuous.
Define $\overline{\delta}:R\to M$ by $\overline{\delta}(x)=[\delta(x)]x$.
For all $x\in R$, we have $\overline{\delta}(x)=[\delta(x)]x=c^{-1}[\eta(x)]x=c^{-1}[\,\overline{\eta}(x)]$.
So, since $\overline{\eta}$ is $C^\infty$,
$\overline{\delta}$ is $C^\infty$ as well.
Also, we have $\overline{\delta}(p)=[\delta(p)]p=ep=p$.

Choose $G_0$, $M_0$, $T$ and $\gamma$ as in \lref{lem-flow-box-existence};
then $G_0T=M_0$,
\begin{itemize}
\item[]\quad$G_0$ is an open neighborhood in $G$ of $e$,
\item[]\quad$M_0$ is an open neighborhood in $M$ of $p$,
\item[]\quad$\gamma:M_0\to G_0$ is $C^\infty$, \qquad $\gamma(p)=e$ \qquad and
\item[]\quad$\forall x\in M_0,\,\,[(\gamma(x))^{-1}]x\in T$.
\end{itemize}
Let $F:G_0\times T\to M_0$ be defined by $F(g,t)=gt$.
Then, by \lref{lem-flow-box-existence}, we see that $F:G_0\times T\to M_0$ is bijective.
Let $\varepsilon:M_0\cap R\to G$ be defined by~$\varepsilon(x)=[\delta(x)][\gamma(x)]$.
Then $\varepsilon(p)=[\delta(p)][\gamma(p)]=[e][e]=e$.
As $\gamma$~is~$C^\infty$, $\gamma$ is continuous.
As $\delta$ and $\gamma$ are both continuous,
$\varepsilon$~is continuous.
Since $\overline{\delta}$ is $C^\infty$,
$\overline{\delta}$ is continuous.
Then, as $\overline{\delta}(p)=p$ and $\varepsilon(p)=e$,
$\overline{\delta}^*(M_0)$ and $\varepsilon^*(G_0)$ are, respectively,
open neighborhoods in~$R$ of~$p$ and in~$M_0\cap R$ of $p$.
Let $U:=[\overline{\delta}^*(M_0)]\cap[\varepsilon^*(G_0)]$.
Then $U$ is an~open neighborhood in~$M_0\cap R$ of $p$,
and is therefore an open neighborhood in $R$ of $p$.
We wish to prove that $\delta$ is $C^\infty$ on~$U$.

We have $U\subseteq M_0\cap R$.
Also, $U\subseteq\overline{\delta}^*(M_0)$, so $\overline{\delta}(U)\subseteq M_0$.
Define $\phi:U\to G$ by~$\phi(x)=\gamma(\overline{\delta}(x))$.
Since $\gamma$ and $\overline{\delta}$ are $C^\infty$, $\phi$ is $C^\infty$.
We have $U\subseteq M_0\cap R\subseteq M_0$.
Define $\psi:U\to G$ by $\psi(x)=(\gamma(x))^{-1}$.
Since $\gamma$ is~$C^\infty$, $\psi$ is $C^\infty$.
Since $\phi$ and $\psi$ are both $C^\infty$,
it suffices to~show: $\forall x\in U$,
$\delta(x)=[\phi(x)][\psi(x)]$.
Let $x\in U$ be given, let $a:=\delta(x)$,
let $g:=\phi(x)$ and let $h:=\psi(x)$.
We wish to show that $a=gh$.

We have $x\in U\subseteq M_0\cap R\subseteq R$.
Let $y:=\overline{\delta}(x)$.
Then we have $y=\overline{\delta}(x)=[\delta(x)]x=ax$.
Recall that $\overline{\delta}(U)\subseteq M_0$.
We have
$$x\in U\subseteq M_0\cap R\subseteq M_0\quad\hbox{and}\quad y=\overline{\delta}(x)\in\overline{\delta}(U)\subseteq M_0.$$
Then, by our choice of $\gamma$ and $T$,
$[(\gamma(x))^{-1}]x\in T$ and $[(\gamma(y))^{-1}]y\in T$.
We have both $g=\phi(x)=\gamma(\overline{\delta}(x))=\gamma(y)$
and $h=\psi(x)=(\gamma(x))^{-1}$.
Let $t:=g^{-1}y$ and $u:=hx$.
Then $gt=y$ and $h^{-1}u=x$.
Also, $t=g^{-1}y=[(\gamma(y))^{-1}]y\in T$ and $u=hx=[(\gamma(x))^{-1}]x\in T$.
Also, $h^{-1}=\gamma(x)$.
Since $\gamma:M_0\to G_0$,
we see that $\gamma(y)\in G_0$.
By definition of~$\varepsilon$, $\varepsilon(x)=[\delta(x)][\gamma(x)]$.
We have $U\subseteq\varepsilon^*(G_0)$, so $\varepsilon(U)\subseteq G_0$.
Then $ah^{-1}=[\delta(x)][\gamma(x)]=\varepsilon(x)\in\varepsilon(U)\subseteq G_0$
and $g=\gamma(y)\in G_0$.
Since
$$F(ah^{-1},u)\,\,=\,\,ah^{-1}u\,\,=\,\,ax\,\,=\,\,y\,\,=\,\,gt\,\,=\,\,F(g,t),$$
and since $F$ is injective, we get: $ah^{-1}=g$ and $u=t$.
Then $a=gh$.
\end{proof}

\section{Local control of stabilizers\wrlab{sect-local-control}}

Let a Lie group $G$ have a $C^\infty$ action on a manifold $M$.

\begin{lem}\wrlab{lem-uniform-discreteness}
Let $p_0\in M$.
Assume that $\Stab_G(p_0)$ is discrete.
Let $\bfu:=\{1_G\}$.
Then there exist open neighborhoods $P$ in $M$ of $p_0$ and $Q$~in~$G$ of~$1_G$
such that, for all $x\in P$, we have $\Stab_Q(x)=\bfu$.
\end{lem}

Let $\scrc$ be the set of all closed subsets of~$G$.
Give $\scrc$ topology of~Hausdorff convergence on compact sets.
Then $\forall$sequence $g_1,g_2,\ldots\to1_G$ in~$G$,
if ( $\forall i\in\N$, $C_i:=$ the closure of the cyclic subgroups generated by $g_i$ ),
then there is a subsequence of $C_1,C_2,\ldots$ that converges, in~$\scrc$, to a
nondiscrete closed subgroup of $G$.
Also, $x\mapsto\Stab_G(x):M\to\scrc$ is set-theoretically upper semicontinuous.
In \lref{lem-uniform-discreteness},
$p_0$ has discrete stabilizer,
so no nondiscrete subgroup of $G$ is contained in that stabilizer.
Upper semicontinuity then imposes restrictions on the stabilizers of points near $p_0$.
Details are as follows.

\begin{proof}
Let $\Lg:=T_{1_G}G$.
Let $\exp:\Lg\to G$ be the Lie theoretic exponential map.
Let $\|\bullet\|$ be a norm on the vector space $\Lg$.
For all $r>0$, define $\Lg_r:=\{Y\in\Lg\,\,\hbox{s.t.}\,\|Y\|<r\}$,
$G_r:=\exp(\Lg_r)$ and $e_r:=\exp|\Lg_r:\Lg_r\to G_r$.
Fix $\varepsilon>0$ such that $G_\varepsilon$~is open in $G$
and $e_\varepsilon:\Lg_\varepsilon\to G_\varepsilon$ is a $C^\infty$ diffeomorphism.
As $\Stab_G(p)$ is discrete, choose an open neighborhood $W$ in $G$ of $1_G$
such that $W\cap(\Stab_G(p_0))=\bfu$.
Choose $\delta>0$ such that $\Lg_{3\delta}\subseteq e_\varepsilon^*(W\cap G_\varepsilon)$.
We have $\Lg_{3\delta}\subseteq e_\varepsilon^*(G_{\varepsilon})=\Lg_\varepsilon$,
so $3\delta\le\varepsilon$.
Let $e:=e_{3\delta}$.
Then $e_\varepsilon|\Lg_{3\delta}=(\exp|\Lg_\varepsilon)|\Lg_{3\delta}
=\exp|\Lg_{3\delta}=e_{3\delta}=e$.
Then $e:\Lg_{3\delta}\to G_{3\delta}$ is a~$C^\infty$~diffeomorphism.
Also, $e_\varepsilon(\Lg_{3\delta})=e(\Lg_{3\delta})$ and $e(0_\Lg)=1_G$.

Let $P_1,P_2,\ldots$ be a sequence of open neighborhoods in $M$ of $p_0$
such that $P_1\supseteq P_2\supseteq\cdots$
and such that $\{P_1,P_2,\ldots\}$ is a neighborhood base in~$M$ at $p_0$.
Let $Q:=G_\delta$.
Then $Q$ is a neighborhood in $G$ of $1_G$.
We wish to show that there exists $j\in\N$ such that,
for all $x\in P_j$, we have $\Stab_Q(x)=\bfu$.
Assume, for a contradiction, that:
for all $j\in\N$, there exists $x\in P_j$ such that $\Stab_Q(x)\ne\bfu$.

For all $j\in\N$, choose $p_j\in P_j$ such that $\Stab_Q(p_j)\ne\bfu$;
then choose $q_j\in(\Stab_Q(p_j))\backslash\bfu$.
Then, since $P_1,P_2,\ldots$ is a semidecreasing neighborhood base in $M$ at $p_0$,
we conclude that $p_1,p_2,\ldots\to p_0$ in $M$.
For all~$j\in\N$, we have $q_j\in Q$ and $q_j\in\Stab_G(p_j)$ and $q_j\ne1_G$.

Since $e:\Lg_{3\delta}\to G_{3\delta}$ is a $C^\infty$ diffeomorphism,
the inverse mapping $e^{-1}:G_{3\delta}\to\Lg_{3\delta}$ is $C^\infty$.
As $e(\Lg_\delta)=\exp(\Lg_\delta)=G_\delta$,
we get $e^{-1}(G_\delta)=\Lg_\delta$.
For all $j\in\N$, let $Y_j:=e^{-1}(q_j)$; then $Y_j\in e^{-1}(Q)=e^{-1}(G_\delta)=\Lg_\delta$.
For all $j\in\N$, $e(Y_j)=q_j\ne1_G=e(0_\Lg)$, so $Y_j\ne0_\Lg$,
so there exists $k\in\N$ such that $2^kY_j\notin\Lg_\delta$.
For all $j\in\N$, let $k_j:=\min\{k\in\N\,|\,2^kY_j\notin\Lg_\delta\}$,
let $\ell_j:=2^{k_j}$, $Z_j:=\ell_jY_j$.
Then, for all $j\in\N$,
$Z_j=\ell_jY_j=2^{k_j}Y_j\notin\Lg_\delta$,
and, as $(1/2)Z_j=(1/2)2^{k_j}Y_j=2^{k_j-1}Y_j\in\Lg_\delta$,
we get $Z_j\in2\Lg_\delta=\Lg_{2\delta}$.

Let $C:=\{Y\in\Lg\hbox{ s.t.}~\delta\le\|Y\|\le2\delta\}$.
Then $C$ is a compact subset of~$\Lg$ and $\Lg_{2\delta}\backslash\Lg_\delta\subseteq C\subseteq\Lg_{3\delta}\backslash\Lg_\delta$.
We have $Z_1,Z_2,\ldots\in\Lg_{2\delta}\backslash\Lg_\delta\subseteq C$.
Fix $Z_0\in C$ such that $Z_0$ is
the limit in $\Lg$ of a convergent subsequence of~$Z_1,Z_2,\ldots$.
For all $j\in\N_0$, $Z_j\in C\subseteq\Lg_{3\delta}\backslash\Lg_\delta\subseteq\Lg_{3\delta}$.
For all $j\in\N_0$, let $w_j:=e(Z_j)$.
Then $w_0$ is the limit in $G$ of a convergent subsequence of~$w_1,w_2,\ldots$.
Also,
$w_0=e(Z_0)\in e(\Lg_{3\delta})=e_\varepsilon(\Lg_{3\delta})\subseteq
e_\varepsilon(e_\varepsilon^*(W\cap G_\varepsilon))\subseteq W\cap G_\varepsilon\subseteq W$.

For all $j\in\N$, we have $\exp(Y_j)=e(Y_j)=q_j$,
so $\exp(\ell_jY_j)=q_j^{\ell_j}$;
then $w_j=e(Z_j)=e(\ell_jY_j)=\exp(\ell_jY_j)=q_j^{\ell_j}$,
so, as $q_j\in\Stab_G(p_j)$, we get $w_j\in\Stab_G(p_j)$,
so $w_jp_j=p_j$.
So, since $w_0$ is the limit in~$G$ of~a convergent subsequence of $w_1,w_2,\ldots$
and since $p_1,p_2,\ldots\to p_0$ in~$M$,
we get $w_0p_0=p_0$.
Then $w_0\in W\cap(\Stab_G(p_0))=\bfu=\{1_G\}$, so~$w_0=1_G$.
As $e(Z_0)=w_0$ and $1_G=e(0_\Lg)$, we get $Z_0=e^{-1}(w_0)$ and $e^{-1}(1_G)=0_\Lg$.
Then $Z_0=e^{-1}(w_0)=e^{-1}(1_G)=0_\Lg$.
Then $0_\Lg=Z_0\in C\subseteq\Lg_{3\delta}\backslash\Lg_\delta$,
and so $0_\Lg\notin\Lg_\delta$.
Then $0=\|0_\Lg\|\ge\delta>0$.
Contradiction.
\end{proof}

\begin{cor}\wrlab{cor-one-or-fewer}
For any subset $S$ of $G$, let $\overline{S}$ denote the closure in $G$ of $S$.
For any subset $S$ of $M$, let $\overline{S}$ denote the closure in $M$ of~$S$.
Let $p_0\in M$ and assume that $\Stab_G(p_0)$ is discrete.
Then there exist open neighborhoods $M_0$ in $M$ of~$p_0$ and $N$ in $G$ of~$1_G$
such that, for all~$x\in\overline{M_0}$, for all $g\in G$,
we have $\#[\Stab_{g\overline{N}}(x)]\le1$.
\end{cor}

\begin{proof}
Choose $P$, $Q$ as in \lref{lem-uniform-discreteness}.
Choose an~open neighborhood $M_0$ in $M$ of $p_0$
such that $\overline{M_0}\subseteq P$.
Choose an open neighborhood $N$~in~$G$ of $1_G$
such that $\overline{N}^{-1}\overline{N}\subseteq Q$.
Let $x\in\overline{M_0}$, $g\in G$ be given.
We wish to show: $\#[\Stab_{g\overline{N}}(x)]\le1$.
Let $a,b\in\Stab_{g\overline{N}}(x)$ be given.
We wish to~show: $a=b$.
Let $c:=b^{-1}a$, $\bfu:=\{1_G\}$.
We wish to show: $c\in\bfu$.

As $a,b\in\Stab_G(x)$, we get $c\in\Stab_G(x)$.
As $a,b\in g\overline{N}$, we get
$c=b^{-1}a\in[g\overline{N}\,]^{-1}[g\overline{N}\,]=\overline{N}^{-1}\overline{N}\subseteq Q$.
As $x\in\overline{M_0}\subseteq P$, by \lref{lem-uniform-discreteness}, $\Stab_Q(x)=\bfu$.
Then $c\,\,\in\,\,(\Stab_G(x))\cap Q\,\,=\,\,\Stab_Q(x)\,\,=\,\,\bfu$.
\end{proof}

\section{Baire points of density\wrlab{sect-Baire-points-density}}

\begin{lem}\wrlab{lem-pts-density}
Let $M$ be a second countable topological space.
Let $L$~be a nonmeager subset of $M$.
Then there exists $p_0\in L$ such that,
for any~neighborhood $U$ in $M$ of $p_0$,
we have: $L\cap U$ is nonmeager in $M$.
\end{lem}

\begin{proof}
Assume, for a contradiction, that, for all $p\in L$,
there exists a neighborhood $U$ in $M$ of $p$ such that $L\cap U$ is meager in $M$.

Let $\scrb$ be a countable basis for the topology  on $M$.
For every $p\in L$, choose a neighborhood $U_p$ in $M$ of $p$
such that $L\cap U_p$ is meager in $M$,
and then choose $B_p\in\scrb$ such that $p\in B_p\subseteq U_p$.
For all $p\in L$, we have $L\cap B_p\subseteq L\cap U_p$, so $L\cap B_p$ is meager in $M$.
Let $\scrz:=\{L\cap B_p\,|\,p\in L\}$.
For all $Z\in\scrz$, $Z$ is meager in $M$.
Let $\scrb':=\{B_p\,|\,p\in L\}$.
Since $\scrb$~is countable and since $\scrb'\subseteq\scrb$,
we conclude that $\scrb'$ is countable.
So, since the map $B\mapsto L\cap B:\scrb'\to\scrz$ is surjective,
$\scrz$ is countable as well.
Then $\bigcup\scrz$ is a countable union of meager subsets of $M$,
and so $\bigcup\scrz$ is meager in $M$.
So, as $L$ is nonmeager in $M$,
we get $L\not\subseteq\bigcup\scrz$.
Choose $p\in L$ such that $p\notin\bigcup\scrz$.
We have $p\in B_p$ and $p\in L$, so $p\in L\cap B_p$.
Since $L\cap B_p\in\scrz$,
we see that $L\cap B_p\subseteq\bigcup\scrz$.
Then $p\in L\cap B_p\subseteq\bigcup\scrz$.

We therefore have both $p\notin\bigcup\scrz$ and $p\in\bigcup\scrz$.
Contradiction.
\end{proof}

It is not hard to improve \lref{lem-pts-density} to show
that there is a meager subset $Z$ of $M$
such that, for all $p\in L\backslash Z$,
for any neighborhood $U$ in~$M$ of $p$,
we have: $L\cap U$ is nonmeager in $M$.
In other words, in any nonmeager set,
the Baire points of density are generic.

\section{Order one stabilizers are generically trivial\wrlab{sect-ordone-stabs-nontriv}}

Let a Lie group $G$ have a $C^\infty$ action on a manifold $M$.
For all $x\in M$, let $G'_x:=\{g\in G\,|\,\forall v\in T_xM,gv=v\}\subseteq\Stab_G(x)$.
Let $\bfu:=\{1_G\}$.

\begin{thm}\wrlab{thm-ordone-stabs-nontriv}
Assume that the $G$-action on $M$ is locally free and fixpoint rare.
Let $L:=\{x\in M\,|\,G'_x\ne\bfu\}$.
Then $L$ is meager in $M$.
\end{thm}

Following J.~Pojhanpelto,
since the trivial action of $\Z$ on~$\R$ has only finite orbits,
each of its suspension $\R$-actions has only closed orbits;
choosing carefully the suspension function,
he ensures that there are nonzero real numbers that
fix, to high order, points in~the suspension.
Attempts to modify this construction
to~disprove \tref{thm-ordone-stabs-nontriv} failed;
an analysis of that failure led to the following proof.
Specifically, the computation in Claim 4 below
is related to the observation that, if the derivative of a
suspension function vanishes on a nonmeager set, then
the function is constant on an open set,
so there is a nonzero element of~$\R$ that
fixes every point in a nonempty open subset of the suspension.
\tref{thm-ordone-stabs-nontriv} is a generalization of this
elementary observation.

\begin{proof}
Assume, for a contradiction, that $L$ is nonmeager in $M$.

For all $S\subseteq G$, let $\overline{S}$~be the closure in $G$ of $S$.
For all $S\subseteq M$, let $\overline{S}$~be the closure in $M$ of $S$.
By \lref{lem-pts-density},
choose $p_0\in L$ such that,
for any~open neighborhood $U$ in $M$ of $p_0$, we have: $L\cap U$ is nonmeager in~$M$.
Since the $G$-action on $M$ is locally free,
it follows that $\Stab_G(p_0)$ is discrete.
By \cref{cor-one-or-fewer},
choose open neighborhoods $M_0$ in $M$ of~$p_0$ and $N$ in $G$ of $1_G$
such that: $\forall x\in\overline{M_0},\forall g\in G$,
$\#[\Stab_{g\overline{N}}(x)]\le1$.
Because of our choice of $p_0$, we see that $L\cap M_0$ is nonmeager in~$M$.

Let $\scrn$ denote the set of open neighborhoods in $G$ of $1_G$
and let $\scrk:=\{K\in\scrn\,|\,\overline{K}\hbox{ is a compact subset of }G\}$.
Let $G^\times:=G\backslash\bfu$.
Then, since the set $\{cK\,|\,c\in G^\times,K\in\scrk,K\subseteq N,c\overline{K}\subseteq G^\times\}$
is an open cover of~$G^\times$
and since $G^\times$ is Lindel\"of,
we choose a sequence $c_1,c_2,\ldots$ in $G^\times$
and a sequence $K_1,K_2,\ldots$ in $\scrk$ such that both
\begin{itemize}
\item[] ( for all $i\in\N$, [ both ( $K_i\subseteq N$ ) and ( $c_i\overline{K_i}\subseteq G^\times$ ) ] ) \qquad and
\end{itemize}
$\displaystyle{\bigcup_{i\in\N}\,c_iK_i=G^\times}$.
Then $\displaystyle{G^\times=\bigcup_{i\in\N}\,c_iK_i\subseteq\bigcup_{i\in\N}\,c_i\overline{K_i}\subseteq G^\times}$,
so $\displaystyle{\bigcup_{i\in\N}\,c_i\overline{K_i}=G^\times}$.
For all $i\in\N$, let
$S_i:=\{x\in\overline{M_0}\,|\,G'_x\cap(c_i\overline{K_i})\ne\emptyset\}$.
Then
\begin{eqnarray*}
\bigcup_{i\in\N}\,S_i&=&
\{x\in\overline{M_0}\,|\,G'_x\cap G^\times\ne\emptyset\}\,\,=\,\,
\{x\in\overline{M_0}\,|\,G'_x\ne\bfu\}\\
&=&\{x\in M\,|\,G'_x\ne\bfu\}\cap\overline{M_0}\,\,=\,\,
L\cap \overline{M_0}\,\,\supseteq\,\,L\cap M_0.
\end{eqnarray*}
So, since $L\cap M_0$ is nonmeager in $M$,
we see that $\displaystyle{\bigcup_{i\in\N}\,S_i}$ is nonmeager in~$M$.
Choose $i_0\in\N$ such that $S_{i_0}$ is nonmeager in $M$.
Let $S:=S_{i_0}$, let $K:=K_{i_0}$ and let $c:=c_{i_0}$.
Then $S$ is nonmeager in~$M$.
Also, $K\in\scrk$, so $\overline{K}$ is a compact subset of $G$,
so $c\overline{K}$ is a compact subset of $G$.
Also, $K\subseteq N$ and $c\overline{K}\subseteq G^\times$.
Also, we have $S=\{x\in\overline{M_0}\,|\,G'_x\cap(c\overline{K})\ne\emptyset\}$.

{\it Claim 1:}
$S$ is closed in $M$.
{\it Proof of Claim 1:}
Let a sequence $s_1,s_2,\ldots$ in $S$ and $x\in M$ be given.
Assume $s_1,s_2,\ldots\to x$ in $M$.
We wish to~show that $x\in S$.
We have $s_1,s_2,\ldots\in S\subseteq\overline{M_0}$,
so, as $\overline{M_0}$ is closed in~$M$,
we conclude that $x\in\overline{M_0}$.
It remains to show that $G'_x\cap(c\overline{K})\ne\emptyset$.

For all $j\in\N$, we have $s_j\in S$,
so $G'_{s_j}\cap(c\overline{K})\ne\emptyset$,
so choose $g_j\in G'_{s_j}\cap(c\overline{K})$.
As $c\overline{K}$ is compact
and as $g_1,g_2,\ldots\in c\overline{K}$,
choose $g\in c\overline{K}$ such that
$g$ is the limit in $G$ of a convergent subsequence of~$g_1,g_2,\ldots$.
It suffices to prove: $g\in G'_x\cap(c\overline{K})$.
As $g\in c\overline{K}$,
we need only show: $g\in G'_x$.
Let $v\in T_xM$ be given.
We wish to prove: $gv=v$.

As $s_1,s_2,\ldots\to x$ in $M$ and $v\in T_xM$,
choose a sequence $v_1,v_2,\ldots$ in~$TM$
such that:
$(\,\,\forall j\in\N, v_j\in T_{s_j}M\,\,)$ and $(\,\,v_1,v_2,\ldots\to v\hbox{ in }TM\,\,)$.
For all $j\in\N$, as $g_j\in G'_{s_j}$ and $v_j\in T_{s_j}M$, we see that $g_jv_j=v_j$.
So, since $g$~is the limit in $G$ of a convergent subsequence of $g_1,g_2,\ldots$
and since $v_1,v_2,\ldots\to v$ in $TM$,
we get $gv=v$.
{\it End of proof of Claim 1.}

{\it Claim 2:}
For all $x\in S$,
$\#[G'_x\cap(c\overline{K})]=1$.
{\it Proof of Claim 2:}
Recall that
$S=\{x\in\overline{M_0}\,|\,G'_x\cap(c\overline{K})\ne\emptyset\}$.
Let $x\in S$ be given.
Then $G'_x\cap(c\overline{K})\ne\emptyset$, so it suffices to prove
that $\#[G'_x\cap(c\overline{K})]\le1$.

We have $x\in S\subseteq\overline{M_0}$,
so, by our choice of $M_0$ and $N$, we see that $\#[\Stab_{c\overline{N}}(x)]\le1$.
Since $G'_x\subseteq\Stab_G(x)$ and $K\subseteq N$,
we see that $G'_x\cap(c\overline{K})\subseteq(\Stab_G(x))\cap(c\overline{N})$,
{\it i.e.}, that $G'_x\cap(c\overline{K})\subseteq\Stab_{c\overline{N}}(x)$.
Then $\#[G'_x\cap(c\overline{K})]\le\#[\Stab_{c\overline{N}}(x)]\le1$.
{\it End of proof of Claim 2.}

Let $R$ denote the interior in $M$ of $S$.
By Claim 1, $S$ is closed in~$M$.
So, since $S$ is nonmeager in $M$,
we conclude that $R\ne\emptyset$.
Then $R$~is a nonempty open subset of $M$.
Since $R\subseteq S$, by Claim 2, define a~function $\eta:R\to G$ by:
for all $x\in R$, $\{\eta(x)\}=G'_x\cap(c\overline{K})$.
Then $\eta(R)\subseteq c\overline{K}$.
Also, for all $x\in R$, we have $\eta(x)\in G'_x$.

{\it Claim 3:}
$\eta:R\to G$ is continuous.
{\it Proof of Claim 3:}
Let a sequence $x_1,x_2,\ldots$ in $R$ and $x_0\in R$ be given.
Assume that $x_1,x_2,\ldots\to x_0$ in~$R$.
We wish to show that $\eta(x_1),\eta(x_2),\ldots\to\eta(x_0)$ in $G$.

Let an open neighborhood $Q$ in $G$ of $\eta(x_0)$ be given,
and then define $I:=\{i\in\N\,|\,\eta(x_i)\in G\backslash Q\}$.
We wish to show that $I$ is finite.
Assume, for a contradiction, that $I$ is infinite.

Fix a sequence $i_1,i_2,\ldots$ in $I$
such that $i_1<i_2<\cdots$.
For all~$j\in\N$, $i_j\in I$, so $\eta(x_{i_j})\in G\backslash Q$.
For all $j\in\N$, let $y_j:=x_{i_j}$ and $g_j:=\eta(y_j)$;
then $g_j\in G\backslash Q$.
As $x_1,x_2,\ldots\to x_0$ in $R$ and
as $y_1,y_2,\ldots$ is a subsequence of~$x_1,x_2,\ldots$,
we get $y_1,y_2,\ldots\to x_0$ in $R$.
Then $y_1,y_2,\ldots\to x_0$ in~$M$.

Since $Q$ is an open neighborhood in $G$ of $\eta(x_0)$,
it follows that $G\backslash Q$~is closed in $G$ and that $\eta(x_0)\in Q$.
For all $j\in\N$, we have $g_j=\eta(y_j)\in G'_{y_j}$.
Also, $g_1,g_2,\ldots\in\eta(R)\subseteq c\overline{K}$.
So, since $c\overline{K}$ is compact,
choose $g_0\in c\overline{K}$ such that
$g_0$ is the limit of a convergent subsequence of $g_1,g_2,\ldots$.
Then, since $g_1,g_2,\ldots\in G\backslash Q$,
and since $G\backslash Q$ is closed in $G$,
it follows that $g_0\in G\backslash Q$.
So, since $\eta(x_0)\in Q$, we conclude that $g_0\ne\eta(x_0)$.
Then $g_0\notin\{\eta(x_0)\}=G'_{x_0}\cap(c\overline{K})$.
So, since $g_0\in c\overline{K}$, it follows that $g_0\notin G'_{x_0}$.
Then choose $v_0\in T_{x_0}M$ such that $g_0v_0\ne v_0$.

As $y_1,y_2,\ldots\to x_0$ in $M$ and $v_0\in T_{x_0}M$,
choose a sequence $v_1,v_2,\ldots$ in $TM$
such that:
$(\,\,\forall j\in\N, v_j\in T_{y_j}M\,\,)$\,\,and\,\,$(\,\,v_1,v_2,\ldots\to v_0\hbox{ in }TM\,\,)$.
For all $j\in\N$, as $g_j\in G'_{y_j}$ and $v_j\in T_{y_j}M$,
we get $g_jv_j=v_j$.
So, since $g_0$ is the limit in $G$ of a convergent subsequence of $g_1,g_2,\ldots$
and since $v_1,v_2,\ldots\to v_0$ in $TM$,
we get $g_0v_0=v_0$.
We therefore have both $g_0v_0\ne v_0$ and $g_0v_0=v_0$.
Contradiction.
{\it End of proof of Claim~3.}

For all $x\in M$, define $f^x:G\to M$ by $f^x(g)=gx$.
For all $u\in TG$, for all $x\in M$,
define $ux:=(df^x)(u)\in TM$.
For all $g\in G$, define $\phi^g:M\to M$ by $\phi^g(x)=gx$.
Then, for all~$g\in G$, for all $v\in TM$, we have $(d\phi^g)(v)=gv$.
For all $x\in M$, let $0_x:=0_{T_xM}$ be the zero vector in~$T_xM$.
Also, for all~$g\in G$, let $0_g:=0_{T_gG}$ be the zero vector in $T_gG$.

Define $F:G\times M\to M$ by $F(g,x)=gx$.
For all $g\in G$, for all $x\in M$, $F(g,x)=f^x(g)$.
Then, for all $u\in TG$, for all $x\in M$,
$(dF)(u,0_x)=(df^x)(u)$.
For all $g\in G$, for all $x\in M$, $F(g,x)=\phi^g(x)$.
Then, for all $g\in G$, for all $v\in TM$,
$(dF)(0_g,v)=(d\phi^g)(v)$.
Therefore, for all $g\in G$, for all~$x\in M$,
for all $u\in T_gG$, for all $v\in T_xM$, we have
\begin{eqnarray*}
(dF)(u,v)&=&[(dF)(u,0_x)]+[(dF)(0_g,v)]\\
&=&[(df^x)(u)]+[(d\phi^g)(v)]\,\,=\,\,ux+gv.
\end{eqnarray*}

Since $R$ is an open subset of $M$,
for all $x\in R$, we identify $T_xR$ with~$T_xM$.
We let $\iota:R\to M$ denote the inclusion map, defined by~$\iota(x)=x$.
Then, for all $x\in R$, for all $v\in T_xM$, we have $(d\iota)(v)=v$.

Let $\overline{\eta}:R\to M$ be defined by~$\overline{\eta}(x)=[\eta(x)]x$.
For all $x\in R$, since $\eta(x)\in G'_x\subseteq\Stab_G(x)$,
we get $\overline{\eta}(x)=x$.
Then $\overline{\eta}=\iota$, so $\overline{\eta}$ is~$C^\infty$.
By~Claim 3, $\eta:R\to G$ is continuous.
Then, by \lref{lem-regularity}, $\eta:R\to G$ is~$C^\infty$.
Define $H:R\to G\times R$ by $H(x)=(\eta(x),x)$.
Then, for all~$x\in R$, for all~$v\in T_xM$,
we get $(dH)(v)=((d\eta)(v),v)$.
Also, for all $x\in R$, we have
$F(H(x))=F(\eta(x),x)=[\eta(x)]x=\overline{\eta}(x)=\iota(x)$.
That is, $F\circ H=\iota$.
Then, by the Chain Rule, we see that $(dF)\circ(dH)=d\iota$.

{\it Claim 4:}
$\forall x\in R,\forall v\in T_xM$, $(d\eta)(v)=0_{\eta(x)}$.
{\it Proof of Claim 4:}
Let $x\in R$ and $v\in T_xM$ be given.
Let $g:=\eta(x)$ and $u:=(d\eta)(v)$.
Then $u\in T_gG$, and we wish to prove that $u=0_g$.

Since $g=\eta(x)\in G'_x\subseteq\Stab_G(x)$,
it follows that $gx=x$.
Then $f^x(g)=gx=x$.
The $G$-action on $M$ is locally free, so $\Stab_G(x)$ is discrete.
Then, by \lref{lem-immersiveness}
(with $f$~replaced by $f^x$, $a$ by $g$ and $p$~by~$x$),
the mapping $(df^x)_g:T_gG\to T_xM$ is injective.

By earlier observations, $(dF)(u,v)=ux+gv$ and $u=(d\eta)(v)$
and $(dH)(v)=((d\eta)(v),v)$ and $(dF)\circ(dH)=d\iota$ and $(d\iota)(v)=v$.
Then
\begin{eqnarray*}
ux+gv&=&(dF)(u,v)\,\,=\,\,(dF)((d\eta)(v),v)\,\,=\,\,(dF)((dH)(v))\\
&=&((dF)\circ(dH))(v)\,\,=\,\,(d\iota)(v)\,\,=\,\,v,
\end{eqnarray*}
so $ux=v-gv$.
Since $g=\eta(x)\in G'_x$ and $v\in T_xM$,
we get $gv=v$.
Then $ux=v-gv=v-v=0_x$.
Since $u\in T_gG$, $(df^x)(u)=(df^x)_g(u)$.
Then $(df^x)_g(u)=(df^x)(u)=ux=0_x=(df^x)_g(0_\Lg)$.
Then, by~injectivity of~$(df^x)_g$, we see that $u=0_\Lg$, as desired.
{\it End of proof of Claim 4.}

By Claim 4, $d\eta$ vanishes,
so $\eta:R\to G$ is constant on~each connected component of~$R$.
Let $W$ be a connected component of $R$,
and choose $g_0\in G$ s.t.~$\eta(W)=\{g_0\}$.
Then
$g_0\in\eta(W)\subseteq\eta(R)\subseteq c\overline{K}\subseteq G^\times$.
Since $R$ is a nonempty open subset of $M$,
$W$ is also a~nonempty open subset of~$M$.
Then, since the $G$-action on $M$ is fixpoint rare,
and since $g\in G^\times=G\backslash\bfu=G\backslash\{1_G\}$,
choose $w\in W$ s.t.~$g_0w\ne w$.
We have $\eta(w)\in\eta(W)=\{g_0\}$, so $\eta(w)=g_0$.
Then $g_0=\eta(w)\in G'_w\subseteq\Stab_G(w)$,
so $g_0w=w$.
Then $w=g_0w\ne w$.
Contradiction.
\end{proof}

\section{Generic freeness in higher order frame bundles\wrlab{sect-genfree-frame}}

Let a Lie group $G$ have a $C^\infty$ action on a manifold~$M$.

\begin{thm}\wrlab{thm-genfree-frame}
Assume that the $G$-action on $M$ is fixpoint rare.
Let $n:=\dim G$.
Then there exists a $G$-invariant meager subset $Z$ of~$\scrf^nM$
such that the $G$-action on $(\scrf^nM)\backslash Z$ is free.
\end{thm}

\begin{proof}
For all $k\in\N_0$,
let $\pi_k:\scrf^kM\to M$ denote the structure map of $\scrf^kM$.
Let $\bfu:=\{1_G\}$.
For all $g\in G\backslash\bfu$, let $Z_g:=\{q\in M\,|\,gq=q\}$;
as the $G$-action on $M$ is fixpoint rare and continuous,
$Z_g$ is nowhere dense and closed in $M$.
If $n=0$, then $G$ is countable, and,
identifying $\scrf^0M$ with~$M$,
we may set $\displaystyle{Z:=\bigcup\,\{Z_g\,|\,g\in G\backslash\bfu\}}$.
We assume $n\ge1$.

Let $\ell:=n-1$.
Let $G^\circ$ be the identity component of $G$.
Since the $G$-action on $M$ is fixpoint rare,
it follows that the $G^\circ$-action on $M$ is fixpoint rare.
Then, by \tref{thm-loc-free-generic},
choose a $G$-invariant dense open subset $M_1$ of $M$
such that the $G$-action on~$\pi_\ell^*(M_1)$ is locally free.

Let $X:=\scrf^nM$, $W:=\scrf^\ell M$, $X_1:=\pi_n^*(M_1)$ and $W_1:=\pi_\ell^*(M_1)$.
Then the $G$-action on $W_1$ is locally free.
We wish to show: there exists a $G$-invariant meager subset $Z$ of~$X$
s.t.~the $G$-action on $X\backslash Z$ is free.

Since $M_1$ is an open subset of $M$ and $\pi_\ell:W\to M$ is continuous,
we see that $W_1$ is an open subset of $W$.
Since $M_1$ is $G$-invariant
and since $\pi_n:X\to M$ is $G$-equivariant,
it follows that $X_1$ is $G$-invariant.
Also, since $M_1$ is dense open in $M$
and since $\pi_n:X\to M$ is open and continuous,
it follows that $X_1$~is dense open in $X$.
Let $Z':=X\backslash X_1$.
Then $Z'$ is a $G$-invariant, nowhere dense, closed subset of~$X$.
For all~$x\in X$, define $G_x:=\Stab_G(x)$.
Let $X_0:=\{x\in X_1\,|\,G_x\ne\bfu\}$.
Then $X_0$ is a $G$-invariant subset of~$X$
and the $G$-action on $X_1\backslash X_0$ is free.
Let $Z:=X_0\cup Z'$.
Then $Z$ is $G$-invariant.
Since, $X\backslash Z'=X_1$,
$X\backslash Z=X_1\backslash X_0$.
Then the $G$-action on $X\backslash Z$ is free.
It remains only to~show that $Z$ is meager in $X$.
So, since $Z'$ is nowhere dense in~$X$ and $Z=X_0\cup Z'$,
it suffices to show that $X_0$ is meager in~$X$.

Let $\sigma:X\to W$ be the canonical map.
For all $w\in W$, we define
$G'_w\,:=\,\{\,g\in G\,|\,\forall v\in T_wW,\,gv=v\,\}$.
Let $L:=\{w\in W_1\,|\,G'_w\ne\bfu\}$.

By assumption, the $G$-action on $M$ is fixpoint rare,
and so, since $\pi_\ell|W_1:W_1\to M$ is both $G$-equivariant and open,
we conclude that the $G$-action on $W_1$ is fixpoint rare.
Then, by~\tref{thm-ordone-stabs-nontriv} (with $M$~replaced by~$W_1$),
$L$~is meager in $W_1$.
So, by $\Leftarrow$ of Corollary 8 of~\cite{adolv:genericframe},
$L$ is meager in $W$.
The map $\sigma:X\to W$ is continuous, open and surjective,
so, by~(i)$\Rightarrow$(ii) of~Corollary 17 of~\cite{adolv:genericframe},
$\sigma^*(L)$~is meager in~$X$.
It therefore suffices to show: $X_0\subseteq\sigma^*(L)$.
Let $x\in X_0$ be given.
We wish to show that $x\in\sigma^*(L)$,
{\it i.e.}, that $\sigma(x)\in L$.
Let $w:=\sigma(x)$.
We wish to show that $w\in L$,
{\it i.e.}, that $w\in W_1$ and $G'_w\ne\bfu$.
As $\pi_\ell\circ\sigma=\pi_n$ and $x\in X_0\subseteq X_1$,
we get $\pi_\ell(\sigma(x))=\pi_n(x)\in\pi_n(X_1)$.
Then
$\pi_\ell(w)=\pi_\ell(\sigma(x))\in\pi_n(X_1)=\pi_n(\pi_n^*(M_1))\subseteq M_1$,
and it follows that $w\in\pi_\ell^*(M_1)$,
{\it i.e.}, that $w\in W_1$.
It remains to show that $G'_w\ne\bfu$.

Since $x\in X_0$, $G_x\ne\bfu$,
so choose $g\in G_x$ such that $g\ne 1_G$.
It suffices to show: $g\in G'_w$.
Let $v\in T_wW$ be given.
We wish to show: $gv=v$.

By \lref{lem-stababovebelow} (with $k$ replaced by $n$),
we have $\Stab_G(x)\subseteq\Stab_G(v)$.
Then $g\in G_x=\Stab_G(x)\subseteq\Stab_G(v)$, so $gv=v$, as desired.
\end{proof}

\begin{cor}\wrlab{cor-genfree-jet}
Assume that the $G$-action on $M$ is fixpoint rare.
Let $d:=\dim M$ and let $p\in\{1,\ldots,d-1\}$.
For all $k\in\N_0$, let $J^k$~denote the bundle of $k$th order
jets of germs of $p$-submanifolds of $M$.
Then there exist $N\in\N_0$ and a $G$-invariant meager subset $Z$ of $J^N$
such that the $G$-action on $J^N\backslash Z$ is free.
\end{cor}

\begin{proof}
For all $x\in M$, for all~$k\in\N_0$,
the action of $\Stab_G(x)$ on $\scrf_x^kM$ is smoothly real algebraic, hence~$C^\omega$.
We can therefore follow the proof of~Theorem 26 of \cite{adolv:genericframe},
except that we use \tref{thm-genfree-frame} in place of~Theorem 20 of \cite{adolv:genericframe}.
\end{proof}


\bibliography{list}

\end{document}